\documentclass[12pt,a4paper]{article}
\usepackage[latin1]{inputenc}
\usepackage{amsmath}
\usepackage{amsfonts}
\usepackage{amssymb}
\usepackage{amsthm}
\usepackage{psfrag}
\usepackage{graphicx}
\usepackage{enumitem}
\usepackage{indentfirst}
\usepackage{bbm}
\newcommand{\lel}{\left\langle}
\newcommand{\rir}{\right\rangle}
\newcommand{\diag}{\text{diag}}
\newtheorem{thm}{Theorem}[section]
\newtheorem{definition}[thm]{Definition}
\newtheorem{lemma}[thm]{Lemma}

\newtheorem{prop}[thm]{Proposition}
\newtheorem{ass}[thm]{Assumption}
\title{Comparison and converse comparison theorems for backward stochastic differential equations with Markov chain noise}
\date{}

\author{Zhe Yang \thanks{Department of Mathematics and Statistics, University of Calgary, 2500 University Drive NW, Calgary, AB, T2N 1N4, Canada, yangzhezhe@gmail.com} \and Dimbinirina Ramarimbahoaka \thanks{Department of Mathematics and Statistics, University of Calgary, 2500 University Drive NW, Calgary, AB, T2N 1N4, Canada, dimbikeli@gmail.com}\and Robert J. Elliott \thanks{Haskayne School of Business, University of Calgary, 2500 University Drive NW, Calgary, AB, T2N 1N4, Canada, School of Mathematical Sciences, University of Adelaide, SA 5005, Australia, relliott@ucalgary.ca}}

\begin{document}
\maketitle
\begin{abstract}
Comparison and converse comparison theorems are important parts of the research on backward stochastic differential equations. In this paper, we obtain comparison results for one dimensional
backward stochastic differential equations with Markov
chain noise, extending and generalizing previous work under natural and simplified hypotheses, and establish a converse comparison theorem for the same type of equation after giving the definition and properties of a type of nonlinear expectation: $f$-expectation.
\end{abstract}
\section{Introduction}
In 1990 Pardoux and Peng \cite{ParPeng1} considered general backward stochastic differential equations (BSDEs for short) in the following form:
$$
Y_t = \xi + \int_t^T g(s, Y_s, Z_s) ds -\int_t^T  Z_sdB_s
,~~~~~t\in[0,T].
$$
Here $B$ is a Brownian Motion and $g$ is the driver or drift of the above BSDE.\\
\indent Since then, comparison theorems of BSDEs have attracted extensive attention. El Karoui, Peng and Quenez \cite{EPQ1997}, Cao
and Yan \cite{Cao99} and Lin \cite{Lin} derived comparison theorems
for BSDEs with Lipschitz continuous coefficients. Liu and Ren
\cite{Liu} proved a comparison theorem for BSDEs with linear growth
and continuous coefficients. Situ \cite{situ} obtained a comparison theorem for BSDEs
with jumps. Zhang \cite{zhang}
deduced a comparison theorem for BSDEs with two reflecting
barriers. Hu and Peng \cite{Hu2006} established a comparison theorem for
multidimensional BSDEs. Comparison theorems for BSDEs have received
much attention because of their importance in applications. For example, the
penalization method for reflected BSDEs is based on a comparison
theorem (see\cite{Cvitanic}, \cite{KKPPQ}, \cite{LS2000} and
\cite{pengxu2005}). Moreover, research on properties of
g-expectations (see, Peng \cite{Peng2004}) and the proof of a
monotonic limit theorem for
BSDEs (see, Peng \cite{Peng99}) both depend on comparison theorems. BSDEs with jumps were also introduced by many. Among others, we mention \cite{barles} and \cite{royer}. Crepey and Matoussi  \cite{crepy1} considered BSDEs with jumps in a more general framework where a Brownian motion is incorporated in the model and a general random measure is used to model the jumps, which in \cite{barles} is a Poisson random measure. \\
\indent It is natural to ask whether the converse of the above results holds or not. That is, if we can compare the solutions of two BSDEs with the same terminal conditions, can we compare the driver?  Coquet, Hu, M\'{e}min and Peng \cite{coquet}, Briand, Coquet, M\'{e}min and Peng \cite{briand}, and Jiang \cite{jiang} derived converse comparison theorems for BSDEs, with no jumps. De Schemaekere \cite{DeSch}, derived a converse comparison theorem for a model with jumps.\\
\indent In 2012, van der Hoek and Elliott \cite{RE1} introduced a
market model where uncertainties are modeled by a finite state
Markov chain, instead of Brownian motion or related jump diffusions,
which are often used when pricing financial derivatives. The
Markov chain has a semimartingale representation involving a vector martingale $M=\{M_t\in\mathbb{R}^N,~t\geq 0\}$. BSDEs in this
framework were introduced by Cohen and Elliott \cite{Sam1} as
$$ Y_t = \xi + \int_t^T f(s, Y_s, Z_s) ds -\int_t^T  Z'_sdM_s
,~~~~~t\in[0,T].
$$
\indent Cohen and Elliott \cite{Sam2} and \cite{Sam3} gave some comparison results for
multidimensional BSDEs in the Markov Chain model
under conditions involving not only the two drivers but also
the two solutions. If we consider two one-dimensional BSDEs driven
by the Markov chain, we extend the comparison result to a situation involving
conditions only on the two drivers.
Consequently our comparison results are easier to use for the one-dimensional case. Moreover, our result in the Markov chain framework needs less conditions on the drivers compared to those in Crepey and Matoussi \cite{crepy1} which are suitable for more general dynamics. \\
\indent Cohen and Elliott \cite{Sam3} also introduced a non-linear expectation: $f$-expectation based on the comparison results in the same paper. Using our comparison results, we shall give $f$-expectation a new definition for one-dimensional BSDEs with Markov chain and show similar properties as those in \cite{Sam3}. Then, we shall provide a converse comparison result for the same model with the use of $f$-expectation. \\
\indent The paper is organized as follows. In Section 2, we introduce the model and give some preliminary results. Section 3 shows our comparison result for one-dimensional BSDEs with Markov chain noise. We introduce the $f$-expectation and give its properties in Section 4. The last section establishes a converse comparison theorem.
\section{The  Model and Some Preliminary Results}\label{prelim}
\indent Consider a finite state Markov chain. Following
\cite{RE1} and \cite{RE2} of van der Hoek and Elliott, we assume the
finite state Markov chain $X=\{X_t, t\geq 0 \}$ is defined on the
probability space $(\Omega,\mathcal{F},P)$ and the state space of
$X$ is identified with the set of unit vectors $\{e_1,e_2\cdots,e_N\}$ in
$\mathbb{R}^N$, where $e_i=(0,\cdots,1\cdots,0) ' $ with 1 in the
$i$-th position. Then the  Markov chain has the semimartingale
representation:
\begin{equation}\label{semimartingale}
X_t=X_0+\int_{0}^{t}A_sX_sds+M_t.
\end{equation}
Here, $A=\{A_t, t\geq 0 \}$ is the rate matrix of the chain $X$ and
$M$ is a vector martingale (See Elliott, Aggoun and Moore
\cite{RE4}).
We assume the elements $A_{ij}(t)$ of $A=\{A_t, t\geq 0 \}$ are bounded. Then the martingale $M$ is square integrable.\\
\indent Take $\mathcal{F}_t=\sigma\{X_s ; 0\leq s \leq t\}$ to be
the $\sigma$-algebra  generated by the Markov process $X=\{X_t\}$
and $\{\mathcal{F}_t\}$ to be its filtration. Since $X$ is right continuous and has left
limits, (written RCLL), the filtration $\{\mathcal{F}_t\}$ is also
right-continuous. The following is given in Elliott \cite{elliott} as Lemma 2.21 :
\begin{lemma}\label{indistinguish}
Suppose $V$ and $Y$ are real valued processes defined on the same probability space $(\Omega, \mathcal{F},P)$ such that for every $t \geq 0$, $V_t = Y_t$, a.s. If both processes are right continuous, then $V$ and $Y$ are indistinguishable, that is:
$$P(V_t = Y_t,~ \text{for any}~ t\geq 0)=1. $$
\end{lemma}
\indent The following product rule for semimartingales can be found in \cite{elliott}.
\begin{lemma}[Product Rule for Semimartingales]\label{ItoPR}
Let $Y$ and $Z$ be two scalar RCLL semimartingales, with no
continuous martingale part. Then
\begin{equation*}
Y_tZ_t = Y_TZ_T - \int_t^T Y_{s_-} dZ_s - \int_t^T Z_{s_-} dY_s -
\sum_{t < s \leq T} \Delta Z_s \Delta Y_s.
\end{equation*}
Here, $\sum\limits_{0< s \leq t} \Delta Z_s \Delta Y_s$ is the optional
covariation of $Y_t$ and $Z_t$ and is also written as $[Z,Y]_t$.
\end{lemma}
For our (vector) Markov chain $X_t \in \{e_1,\cdots,e_N\}$,
note that $X_t X'_t = \diag(X_t)$. Also, $dX_t= A_t X_tdt+ dM_t$. By Lemma \ref{ItoPR}, we know  for $t\in[0,T],$
\begin{align}\label{1}
\nonumber
X_tX'_t &= X_0X'_0 + \int_0^t X_{s-} dX'_s + \int_0^t (dX_s) X'_{s-} + \sum_{0 < s \leq t} \Delta X_s \Delta X'_s \\
\nonumber
  &= \diag(X_0) + \int_0^t X_s (A_sX_s)' ds + \int_0^t X_{s-} dM'_s + \int_0^t A_s X_s X'_{s-} ds\\
\nonumber
& ~~ + \int_0^t (dM_s ) X'_{s-} + [X,X]_t \\
\nonumber
& = \diag (X_0) + \int_0^t X_s X'_s A'_s ds + \int_0^t X_{s-} dM'_s + \int_0^t A_s X_s X'_{s-} ds \\
&~~ + \int_0^t (dM_s) X'_{s-} + [X,X]_t
- \lel X,X\rir_t + \lel X,X \rir_t.
\end{align}
Here, $\lel X, X\rir$ is the unique predictable $N\times N$ matrix process such that
$[X,X]-\lel X,X \rir$ is a matrix valued martingale and write
\begin{equation}\label{L_t}
L_t = [X,X]_t - \lel X,X\rir_t, \quad t \in [0,T].
\end{equation}
However,
\begin{equation}\label{2}
X_tX'_t = \diag (X_t) = \diag(X_0) + \int_0^t \diag (A_s X_s)  ds +
\int_0^t \diag(M_s).
\end{equation}
Equating the predictable terms in \eqref{1} and \eqref{2}, we have
\begin{equation}\label{3}
\lel X, X\rir_t =  \int_0^t \diag(A_sX_s) ds - \int_0^t \diag(X_s)
A'_s ds - \int_0^t A_s \diag(X_s) ds.
\end{equation}
\indent For $n\in\mathbb{N}$, denote for $\phi \in \mathbb{R}^n$, the Euclidean norm $|\phi|_{n}=\sqrt{\phi'\phi}$ and for $\psi \in \mathbb{R}^{n\times n}$, the matrix norm $\|\psi\|_{n\times n}=\sqrt{Tr(\psi'\psi)}$.\\
Let $\Psi$ be the matrix
\begin{equation}\label{Psi}\Psi_t = \diag(A_tX_t)- \diag(X_t)A'_t - A_t \diag(X_t).
\end{equation}
Then $d\langle X,X\rangle_t=\Psi_tdt.$ For any $t>0$, Cohen and Elliott \cite{Sam1,Sam3}, define the semi-norm $\|.\|_{X_t}$, for
$C, D \in \mathbb{R}^{N\times K}$ as :
\begin{align*}
\lel C, D\rir_{X_t} & = Tr(C' \Psi_tD), \\[2mm]
\|C\|^2_{X_t} & = \lel C, C\rir_{X_t}.
\end{align*}
We only consider the case where $C \in \mathbb{R}^N$, hence we
introduce the semi-norm $\|.\|_{X_t}$ as:
\begin{align}\label{normC}
\nonumber
\lel C, D\rir_{X_t}  = C' \Psi_t D, \\[2mm]
\|C\|^2_{X_t}  = \lel C, C\rir_{X_t}.
\end{align}
It follows from equation \eqref{3} that
\[\int_t^T \|C\|^2_{X_s} ds = \int_t^T  C' d\lel X, X\rir_s C.\]
\indent Consider a one-dimensional BSDE with the Markov chain noise
as follows:
\begin{equation}\label{BSDEMC}
Y_t = \xi + \int_t^T f(s, Y_s, Z_s ) ds -\int_t^T  Z'_{s} dM_s
,~~~~~t\in[0,T].
\end{equation}
Here the terminal condition $\xi$ and the coefficient $f$ are known. For $t>0$, denote
\begin{align*}
& L^2(\mathcal{F}_t): =\{ \mathbb{R} \text{-valued}~ \mathcal{F}_t \text{-measurable random variables such that} E[|\xi|^2]< \infty\}.
\end{align*}
\indent Lemma \ref{existence} (Theorem 6.2 in Cohen and Elliott \cite{Sam1})
gives the existence and uniqueness result of solutions to the BSDEs
driven by Markov chains.
\begin{lemma}\label{existence}
Assume $\xi \in L^2(\mathcal{F}_T)$ and the predictable
function $f: \Omega \times [0, T] \times \mathbb{R} \times
\mathbb{R}^N \rightarrow \mathbb{R}$ satisfies a Lipschitz
condition, in the sense that there exists two constants $l_1, l_2>0$  such
that for each $y_1,y_2 \in \mathbb{R}$ and $z_1,z_2 \in
\mathbb{R}^{N}$,
\begin{equation}\label{Lipchl}
|f(t,y_1,z_1) - f(t, y_2, z_2)| \leq l_1 |y_1-y_2| +l_2 \|z_1
-z_2\|_{X_t}.
\end{equation}
We also assume $f$ satisfies
\begin{equation}\label{finite}
 E [ \int_0^T |f(t,0,0)|^2 dt] <\infty.
\end{equation}
 Then there exists a solution $(Y, Z)$
to the BSDE (\ref{BSDEMC}). Moreover, \\[2mm]
(1) $Y$ is an $\mathbb{R}$-valued adapted RCLL process satisfying $
E[\int_0^T |Y_s|^2 ds] <  \infty$;\\[2mm]
(2) $Z$ is a predictable vector process in $\mathbb{R}^N$ satisfying
$  E[\int_0^T \|Z_s\|_{X_s}^2 ds] < \infty; $\\[2mm]
 (3) this solution is
unique up to indistinguishability for $Y$ and equality $d\langle
M,M\rangle_t$ $\times\mathbb{P}$-a.s. for $Z$.
\end{lemma}
The following lemma is an extension result to stopping time of Lemma \ref{existence} (see Cohen and Elliott \cite{Sam3}).
\begin{lemma}\label{BSDEST}
Suppose $\tau >0$ is a stopping time such that there exists a real value $T$ with $P(\tau > T)=0$, $\xi \in L^2(\mathcal{F}_{\tau})$ and $f$ satisfies \eqref{Lipchl} and \eqref{finite}, with integration from $0$ to $\tau$, then the BSDE
\begin{equation}
Y_t = \xi + \int_{t\wedge\tau}^{\tau} f(s, Y_s, Z_s ) ds -\int_{t\wedge\tau}^{\tau}  Z'_{s} dM_s
,~~~~~t\geq 0
\end{equation}
has a unique solution satisfying (1), (2) and (3) of Lemma \ref{existence}, with integration from $0$ to $\tau$.
\end{lemma}
See Campbell and Meyer \cite{campbell} for the following definition:
\begin{definition}[Moore-Penrose pseudoinverse]\label{defMoore}
The Moore-Penrose pseudoinverse of a square matrix $Q$ is the matrix $Q^{\dagger}$ satisfying the properties:\\[2mm]
  1) $QQ^{\dagger}Q = Q$ \\[2mm]
  2) $Q^{\dagger}QQ^{\dagger} = Q^{\dagger}$ \\[2mm]
  3) $(QQ^{\dagger})' = QQ^{\dagger}$ \\[2mm]
  4) $(Q^{\dagger}Q)'=Q^{\dagger}Q.$
\end{definition}
\indent Recall the matrix $\Psi$ given by \eqref{Psi}. We adapt Lemma 3.5 in Cohen and Elliott \cite{Sam3} for our one-dimensional framework as follows:
\begin{lemma}\label{sam35}
For any driver satisfying \eqref{Lipchl} and \eqref{finite}, for any $Y$ and $Z$
\[P(f(t,Y_{t-},Z_t) = f(t,Y_{t-}, \Psi_t\Psi_t^{\dagger} Z_t), ~\text{ for all}~ t \in [0,+\infty])=1\]
and
\[\int_0^t Z'_s dM_s = \int_0^t (\Psi_s\Psi^{\dagger}_s Z_s)' dM_s.\]
Therefore, without any loss of generality, assume  $Z=\Psi\Psi^{\dagger} Z$.
\end{lemma}
\section{A comparison theorem for one-dimensional BSDEs with Markov chain noise}\label{section3}
\begin{ass}\label{ass0}
Assume the Lipschitz constant $l_2$ of the driver $f$ given in \eqref{Lipchl} satisfies  $$~~~~~~l_2\|\Psi_t^{\dagger}\|_{N \times N} \sqrt{6m}\leq 1, ~~~\text{ for any }~t \in [0,T],$$ where $\Psi$ is given in \eqref{Psi} and $m>0$ is the bound of $\|A_t\|_{N\times N}$, for any $t\in[0,T]$.
\end{ass}
\begin{ass}\label{ass00}
Assume the Lipschitz constant $l_2$ of the driver $f$ given in \eqref{Lipchl} satisfies  $$l_2\|\Psi_t^{\dagger}\|_{N \times N} \sqrt{6m}< 1, ~~~\text{ for any }~t \in [0,T],$$ where $\Psi$ is given in \eqref{Psi} and $m>0$ is the bound of $\|A_t\|_{N\times N}$, for any $t\in[0,T]$.
\end{ass}
\indent For $i=1,2,$ suppose $(Y^{(i)},Z^{(i)})$ is the solution of one-dimensional
BSDE with Markov chain noise:
$$Y^{(i)}_t = \xi_i + \int_t^T f_i(s, Y^{(i)}_s, Z^{(i)}_s ) ds
- \int_t^T (Z_{s}^{(i)})' dM_s,\hskip.4cmt\in[0,T].$$
\begin{thm}\label{strict_theorem}
Assume $\xi_1,\xi_2 \in L^2(\mathcal{F}_T)$ and $f_1,f_2:\Omega \times [0,T]\times \mathbb{R}\times \mathbb{R}^N \rightarrow \mathbb{R}$ satisfy some conditions such that the above two BSDEs have unique solutions. Moreover assume $f_1$ satisfies \eqref{Lipchl} and Assumption \ref{ass0}.
If $\xi_1 \leq \xi_2 $, a.s. and $f_1(t,Y_t^{(2)}, Z_t^{(2)}) \leq f_2(t,Y_t^{(2)}, Z_t^{(2)})$, a.e., a.s., then
$$P( Y_t^{(1)}\leq Y_t^{(2)},~~\text{ for any } t \in [0,T])=1.$$
Moreover, if $f_1$ satisfies Assumption \ref{ass00},
$$
Y_0^{(1)}=Y_0^{(2)}
\Longleftrightarrow\left \{\begin{array}{lll}f_1(t,Y_t^{(2)},Z_t^{(2)})=f_2(t,Y_t^{(2)},Z_t^{(2)}),~\text{ a.e.,~a.s.;}\\[2mm]
 \xi_1=\xi_2, ~\text{a.s. } \end{array} \right.
$$
\end{thm}
{\bf Proof.}
Set $Y_t = Y_t^{(2)}- Y_t^{(1)}$, $Z_t = Z_t^{(2)}- Z_t^{(1)}$, $\xi= \xi_2-\xi_1$, $f_s =f_2(s, Y_s^{(2)}, Z_s^{(2)})$ $-f_1(s, Y_s^{(2)}, Z_s^{(2)})$, and define
\begin{align*}
a_s & =
\begin{cases}
\dfrac{f_1(s, Y_s^{(2)}, Z_s^{(2)})-f_1(s, Y_s^{(1)}, Z_s^{(2)})}{Y_s}, \quad ~~~\text{if}\quad Y_s \neq 0 ;\\
0, ~~~~~~~~~~~~~~~~~~~~~~~~~~~~~~~~~~~~~~~~~~~~~~~~\text{if}\quad Y_s=0
\end{cases}
\end{align*}
and
\begin{align*}
b_s & =
\begin{cases}
\dfrac{f_1(s, Y_s^{(1)}, Z_s^{(2)})-f_1(s, Y_s^{(1)}, Z_s^{(1)})}{|Z_s|^2}Z_s', \quad \text{if}\quad Z_s \neq 0; \\
0, ~~~~~~~~~~~~~~~~~~~~~~~~~~~~~~~~~~~~~~~~~~~~~~~~\text{if}\quad Z_s=0.
\end{cases}
\end{align*}
Then, we have:
\begin{align}\label{equivalence}
Y_t = \xi + \int_t^T (a_sY_s+b_sZ_s+f_s) ds - \int_t^T Z'_s dM_s, ~~~ t \in [0,T].
\end{align}
\begin{lemma}[Duality]\label{Duality}
For $t\in [0, T]$, consider the one-dimensional SDE
\begin{equation}\label{sde}
\begin{cases}
dU_s= U_s a_s ds + U_{s-} b_{s-} (\Psi_s^{\dagger})'dM_s, ~~ s \in [t,T];\\
~~U_t =1.
\end{cases}
\end{equation}
Then the solution of the one-dimensional linear BSDE \eqref{equivalence} satisfies
\begin{equation*}
P (Y_t = E [ \xi U_T + \int_t^T f_s U_s ds | \mathcal{F}_t], ~ \text{for any}~ t\in [0,T])=1.
\end{equation*}
\end{lemma}
\begin{proof}
Recall $[M,M]_t=[X,X]_t=\lel X,X \rir_t +L_t$ and $d \lel X,X\rir_t = \Psi_t dt$. Applying Ito's formula on $U_sY_s$, $t \leq s \leq T$, and using Lemma \ref{sam35}, we derive
\begin{align*}
d(U_sY_s) & = U_{s-} dY_s + Y_{s-}dU_s + d[U,Y]_s \\[2 mm]
& = - U_s a_s Y_s ds - U_s b_s Z_s ds - U_s f_s ds + U_{s-}Z'_{s} dM_s  + Y_s U_s a_s ds \\[2 mm]
& ~~ + Y_{s-}U_{s-}b_{s-} (\Psi_s^{\dagger})'dM_s + Z'_s \Delta M_s U_{s-}b_{s-} (\Psi_s^{\dagger})' \Delta M_s\\[2 mm]
& =  - U_s b_s Z_s ds - U_s f_s ds + U_{s-}Z'_{s} dM_s   \\[2 mm]
& ~~ + Y_{s-}U_{s-}b_{s-} (\Psi_s^{\dagger})'dM_s + Z'_s \Delta M_s \Delta M'_s \Psi_s^{\dagger} U_{s-}b'_{s-}  \\[2 mm]
& =  - U_s b_s Z_s ds - U_s f_s ds + U_{s-}Z'_{s} dM_s   \\[2 mm]
& ~~ + Y_{s-}U_{s-}b_{s-} (\Psi_s^{\dagger})'dM_s +  Z'_sd[M ,M]_s\Psi_s^{\dagger}U_{s-}b'_{s-}\\[2 mm]
& = - U_s b_s Z_s ds - U_s f_s ds + U_{s-}Z'_s dM_s  \\[2 mm]
& ~~ + Y_{s-}U_{s-}b_{s-} (\Psi_s^{\dagger})'dM_s +  Z'_s\Psi_s \Psi_s^{\dagger}U_sb'_s ds + Z'_sdL_s \Psi_s^{\dagger} U_{s-}b'_{s-}\\[2mm]
&= - U_s b_s Z_s ds - U_s f_s ds + U_{s-}Z'_s dM_s  \\[2 mm]
& ~~ + Y_{s-}U_{s-}b_{s-} (\Psi_s^{\dagger})'dM_s +  Z'_sU_sb'_s ds + Z'_sdL_s \Psi_s^{\dagger} U_{s-}b'_{s-} \\[2 mm]
& =  - U_s f_s ds + U_{s-}Z'_s dM_s  + Y_{s-}U_{s-}b_{s-} (\Psi_s^{\dagger})'dM_s  + Z'_sdL_s \Psi_s^{\dagger} U_{s-}b'_{s-}.
\end{align*}
Integrating both sides of above equation from $t$ to $T$ and taking the expectation given $\mathcal{F}_t$, we deduce for any $t \in [0,T]$,
$$
Y_t = E [ \xi U_T + \int_t^T f_s U_s ds |\mathcal{F}_t], ~~~ \text{a.s.}
$$
Since $Y_{\cdot}$ and $E [ \xi U_T + \int_\cdot^T f_s U_s ds |\mathcal{F}_{\cdot}]$ are both RCLL, by Lemma \ref{indistinguish}, the result holds.
\end{proof}
\begin{lemma}\label{normbound}
For any $C \in \mathbb{R}^N$,
$$ ~~~~\|C\|_{X_t} \leq \sqrt{3m} |C|_N, ~~\text{ for any }t\in[0,T],$$
where $m>0$ is the bound of $\|A_t\|_{N\times N}$, for any $t\in[0,T]$.
\end{lemma}
\begin{proof} Since the elements $A_{ij}(t)$ of $A=\{A_t, t\geq 0 \}$ are bounded, there exists a constant $m>0$ such that $\|A_t\|_{N \times N} \leq m$, for any $t\in[0,T]$.
From the definition in \eqref{normC}, we have:
\begin{align*}
\|C\|^2_{X_t}& \leq  |C|^2_N \cdot\|\diag(A_tX_t)- \diag(X_t)A'_t - A_t \diag(X_t)\|_{N\times N} \\[2mm]
& \leq  |C|^2_N \cdot(\|\diag(A_tX_t)\|_{N\times N}+\|\diag(X_t)A'_t\|_{N\times N} +\| A_t \diag(X_t)\|_{N\times N} )\\[2mm]
& \leq  |C|^2_N \cdot(|A_tX_t|_{ N}+|X_t|_N\cdot\|A_t\|_{N\times N} +\| A_t\|_{N\times N}\cdot|X_t|_N )\\[2mm]
& \leq  |C|^2_N \cdot(\|A_t\|_{N\times N}\cdot|X_t|_{ N}+|X_t|_N\cdot\|A_t\|_{N\times N} +\| A_t\|_{N\times N}\cdot|X_t|_N )\\[2mm]
& \leq  3 |C|^2_N\cdot \|A_t\|_{N\times N} \leq 3 m |C|_N^2~.
\end{align*}
\end{proof}

We go back to the proof of Theorem \ref{strict_theorem}. We follow the notations in Lemma \ref{Duality}. Denote $$dV_s = a_sds + b_{s-} (\Psi_s^{\dagger})'dM_s,~~s\in[0,T].$$ The solution to SDE \eqref{sde} is given by the Dol\'{e}an-Dade exponential (See \cite{elliott}):
\[U_s = \exp(V_s - \frac{1}{2}\lel V^c,V^c\rir_s)\prod_{0 \leq u\leq s} (1+\Delta V_u )e^{-\Delta V_u},~~s\in[0,T],\]
 where $$\Delta V_u = b_{u-} (\Psi_u^{\dagger})' \Delta M_u = b_{u-} (\Psi_u^{\dagger})' \Delta X_u.$$
 If $f_1$ satisfies Assumption \ref{ass0}, we deduce
\begin{align*}
|\Delta V_u| &\leq |b_{u-} |_N \cdot \|(\Psi_u^{\dagger})'\|_{N\times N} \cdot |\Delta X_u|_N \\& \leq l_2 \dfrac{\|Z_u\|_{X_u}}{|Z_u|_N}\dfrac{1}{\sqrt{6m}l_2}\sqrt{2} \\
& \leq \sqrt{3m}l_2 \dfrac{1}{\sqrt{6m}l_2}\sqrt{2}\\
& =1.
\end{align*}
 hence we have $U_s \geq0$ for any $s\in[0,T]$. By Lemma \ref{Duality}, we know for any $t \in [0,T]$,
 $$
Y_t = E [ \xi U_T + \int_t^T f_s U_s ds |\mathcal{F}_t] , ~a.s.
$$
 As $\xi \geq 0$, a.s., and $f_s \geq 0$, a.e., a.s., it follows that for any $t \in [0,T]$, $Y_t \geq 0$,  a.s. Since $Y_{\cdot}$ and $E [ \xi U_T + \int_\cdot^T f_s U_s ds |\mathcal{F}_{\cdot}]$ are both RCLL, by Lemma \ref{indistinguish},
 \[P(Y_t \geq 0, ~ \text{for any } t \in [0,T])=1.\]
  Moreover, if $f_1$ satisfies Assumption \ref{ass00}, then $U_s >0$, for any $s\in[0,T]$. Hence, $$Y_0=0\Longleftrightarrow \xi=0,~ a.s.,~ \text{and}~ f_t=0,~ a.e.,~ a.s.$$
\section{$f$-expectation}
Now we introduce the nonlinear expectation: $f$-expectation. The $f$-expectation, for a fixed driver $f$, is an interpretation of the solution to a BSDE as a type of nonlinear expectation. Here, we give the one-dimensional case of the definitions and properties in Cohen and Elliott \cite{Sam3}, based on our comparison theorems.
\begin{ass}\label{ass1}
Suppose $f:\Omega \times [0,T]\times \mathbb{R}\times \mathbb{R}^N \rightarrow \mathbb{R}$ satisfies \eqref{Lipchl} and \eqref{finite} such that\\[2mm]
(I) For all $(t,y) \in \mathbb{R}\times\mathbb{R}$, $f(t,y,0)=0$, a.s.;\\[2mm]
(II) For all $(y,z) \in \mathbb{R}\times\mathbb{R}^N$, $t \rightarrow f(t,y,z)$ is continuous.
\end{ass}
In this section, we suppose the driver $f$ satisfies Assumption \ref{ass00} and Assumption \ref{ass1}.
Before introducing the $f$-expectation, we shall give the following definition:
\begin{definition}\label{f_evaluation}
For a fixed driver $f$, given $t \in [0, T]$ and $\xi \in L^2(\mathcal{F}_t)$, define for each $ s\in [0,t]$ ,
 $$\mathcal{E}_{s,t}^f(\xi)= Y_s,$$
  where $(Y,Z)$ is the solution of
$$ Y_s = \xi + \int_s^t f(u, Y_u, Z_u ) du -\int_s^t  Z'_{u} dM_u,~~ s\in[0,t].$$
\end{definition}
\begin{prop}\label{f_eval_lem}
$\mathcal{E}_{s,t}^{f}(\cdot)$ defined above satisfies:\\
(1) For any  $\xi \in L^2(\mathcal{F}_s)$, $\mathcal{E}_{s,t}^f(\xi) = \xi$, a.s.\\
(2) If for any $\xi_1,\xi_2 \in L^2(\mathcal{F}_t)$, $\xi_1 \geq \xi_2$, a.s., then $\mathcal{E}_{s,t}^f(\xi_1) \geq \mathcal{E}_{s,t}^f(\xi_2)$. Moreover, $$\mathcal{E}_{s,t}^f(\xi_1) = \mathcal{E}_{s,t}^f(\xi_2)\Longleftrightarrow\xi_1=\xi_2, ~\text{a.s.}$$
(3) For any $r \leq s \leq t$, $\mathcal{E}^f_{r,s}(\mathcal{E}_{s,t}^f(\xi)) =\mathcal{E}^f_{r,t}(\xi)$, a.s.\\
(4) For any $A\in \mathcal{F}_s$, $\mathbb{I}_{A}\mathcal{E}^{f}_{s,t}(\xi) = \mathbb{I}_A \mathcal{E}^{f}_{s,t}(\mathbb{I}_A\xi)$, a.s.
\end{prop}
\begin{proof}
(1) is clear. (2) is the result of Theorem \ref{strict_theorem}. For (3), for any $r \leq s \leq t$, $Y_r$ is the solution at time $r$ of the following BSDE, with terminal time $t$ and terminal value $\xi$,
$$ \begin{array}{lll}Y_r &= \xi + \int_r^t f(u, Y_u, Z_u ) du -\int_r^t  Z'_{u} dM_u\\[2mm]
&=  \xi + \int_s^t f(u, Y_u, Z_u ) du -\int_s^t Z'_{u} dM_u+\int_r^s f(u, Y_u, Z_u ) du -\int_r^s Z'_{u} dM_u\\[2mm]
&=Y_s + \int_r^s f(u, Y_u, Z_u ) du -\int_r^s Z'_{u} dM_u. \end{array}$$
This means that $Y_r$ is also the value of a solution to the BSDE with same driver $f$ and terminal time $s$, that is $\mathcal{E}^f_{r,s}(Y_s)= Y_r$. Then we have
\[\mathcal{E}_{r,s}^f (\mathcal{E}_{s,t}^f(\xi))= \mathcal{E}_{r,s}^f(Y_s)=Y_r = \mathcal{E}_{r,t}^f(\xi).\]
To prove (4), consider the following two BSDEs
$$ Y^{(1)}_s = \xi + \int_s^t f(u, Y^{(1)}_u, Z^{(1)}_u) du - \int_s^t (Z^{(1)}_u)' dM_u,~~s\in[0,t] $$
and
$$ Y^{(2)}_s = \mathbb{I}_A\xi + \int_s^t f(u, Y^{(2)}_u, Z^{(2)}_u) du - \int_s^t (Z^{(2)}_u)' dM_u,~~s\in[0,t].$$
Since $A \in \mathcal{F}_s$, we know by Assumption \ref{ass1} (I) that $\mathbb{I}_A f(u, Y^{(i)}_u, Z^{(i)}_u) = f(u, \mathbb{I}_A Y^{(i)}_u, \mathbb{I}_A Z^{(i)}_u)$, $i=1,2$, moreover,
$$ \mathbb{I}_{A}Y^{(1)}_s = \mathbb{I}_{A}\xi + \int_s^t f(u, \mathbb{I}_A Y^{(1)}_u, \mathbb{I}_A Z^{(1)}_u) du - \int_s^t \mathbb{I}_{A} (Z^{(1)}_u)' dM_u, ~~s\in[0,t]$$
and
$$ \mathbb{I}_{A}Y^{(2)}_s = \mathbb{I}_{A}\xi + \int_s^t f(u, \mathbb{I}_A Y^{(2)}_u, \mathbb{I}_A Z^{(2)}_u) du - \int_s^t \mathbb{I}_{A} (Z^{(2)}_u)' dM_u,~~s\in[0,t].$$
By uniqueness of solution of the BSDE given in Lemma \ref{existence}, it follows that $\mathbb{I}_{A}\mathcal{E}_{s,t}^f(\xi)=\mathbb{I}_A Y_s^{(1)} = \mathbb{I}_A Y_s^{(2)} = \mathbb{I}_A \mathcal{E}_{s,t}^f(\mathbb{I}_A\xi).$
\end{proof}
\begin{definition}\label{f_expectation}
Define, for $\xi \in L^2(\mathcal{F}_T)$ and a driver $f$,
\[\mathcal{E}_f(\xi):= \mathcal{E}^f_{0,T}(\xi), ~\text{and}~~ \mathcal{E}_f(\xi| \mathcal{F}_t):= \mathcal{E}^f_{t,T}(\xi).\]
$\mathcal{E}_f(\xi)$ is called $f$-expectation and $\mathcal{E}_f(\xi| \mathcal{F}_t)$ is called conditional $f$-expectation.
\end{definition}
The following properties follows directly from Definition \ref{f_expectation}, Proposition \ref{f_eval_lem} and Lemma \ref{BSDEST}.
\begin{prop}\label{f_exp_lem}
Let $s,t \leq T$, be two stopping times.\\
(1) For $\xi \in L^2(\mathcal{F}_t)$,  $\mathcal{E}_f(\xi |\mathcal{F}_t) = \xi$, a.s.\\
(2) If for any $\xi_1,\xi_2 \in L^2(\mathcal{F}_T)$,  $\xi_1 \geq \xi_2$, a.s., then $\mathcal{E}_f(\xi_1| \mathcal{F}_t) \geq \mathcal{E}_f(\xi_2| \mathcal{F}_t)$. Moreover,  $\mathcal{E}_f(\xi_1) = \mathcal{E}_f(\xi_2)$ $\Longleftrightarrow$ $\xi_1=\xi_2,$ a.s.\\
(3) For any $ s \leq t$, $\mathcal{E}_f(\mathcal{E}_f(\xi|\mathcal{F}_t)|\mathcal{F}_s) =\mathcal{E}_f(\xi| \mathcal{F}_s)$, a.s. Moreover, $\mathcal{E}_f(\mathcal{E}_f(\xi|\mathcal{F}_s)) =\mathcal{E}_f(\xi)$.\\
(4) For any $A\in \mathcal{F}_t$, $\mathbb{I}_{A}\mathcal{E}_{f}(\xi|\mathcal{F}_t) = \mathbb{I}_A \mathcal{E}_{f}(\mathbb{I}_A\xi|\mathcal{F}_t)$, a.s.
\end{prop}
\section{A Converse Comparison Theorem, for one-dimensional BSDE with Markov chain noise}
Our converse comparison theorem uses the theory of an $f$-expectation in the previous section. For $i=1,2$, consider the BSDEs with same terminal condition $\xi$:
$$Y^{(i)}_t = \xi + \int_t^T f_i(s, Y^{(i)}_s, Z^{(i)}_s ) ds
- \int_t^T (Z_{s}^{(i)})' dM_s,\hskip.4cmt\in[0,T].$$
\begin{thm}
Suppose  $f_1$ satisfies Assumption \ref{ass00}, Assumption \ref{ass1} and $f_2$ satisfies Assumption \ref{ass1}. Then the following are equivalent:\\
 i) For any $\xi  \in L^2(\mathcal{F}_T)$, $\mathcal{E}_{f_1}(\xi) \leq \mathcal{E}_{f_2}(\xi)$; \\
ii) $P(f_1(t,y,z) \leq f_2(t,y,z), ~ \text{for any} ~ (t,y,z)\in [0,T]\times \mathbb{R}\times\mathbb{R}^N) =1.$
\end{thm}
\begin{proof}
ii) $\Rightarrow$ i) is given by Theorem \ref{strict_theorem}.\\
Let us prove i) $\Rightarrow$ ii). For each $\delta > 0$ and $(y,z) \in \mathbb{R}\times \mathbb{R}^N$, introduce the stopping time:
\[\tau_{\delta}=\tau_{\delta}(y,z) = \inf\{t \geq 0;~ f_2 (t, y, z) \leq f_1 (t, y,z)-\delta\} \wedge T.\]
Suppose ii) does not hold,
then there exists $\delta> 0$ and $(y,z) \in \mathbb{R}\times \mathbb{R}^N$ such that  $P(\tau_{\delta}(y,z) < T) >0.$
For $(\delta, y, z)$ such that $P(\tau_{\delta}(y,z) < T) >0$, consider for $i=1,2$, the following SDE
$$\left \{ \begin{array}{ll}
 dY^{i}(t) = - f_i (t, Y^i(t),z) dt + zdM_t, ~~~ t\in [\tau_{\delta}, T], \\[2mm]
 Y^i(\tau_{\delta}) =y.
\end{array} \right.$$
For $i = 1,2$, the above equation admits a unique solution $Y^{(i)}$ (See Elliott\cite{elliott}, Chapter 14). Define:
\begin{align*}
& \tau'_{\delta} = \inf\{t \geq \tau_{\delta};~ f_2(t, Y^{(2)}(t),z) \geq f_1(t, Y^{(1)}(t), z) - \frac{\delta}{2}\}\wedge T,
\end{align*}
with $\tau'_{\delta} = T$ if $\tau_{\delta}=T$.
We know $\Omega= \{\tau_{\delta} \leq \tau'_{\delta}\} = \{\tau_{\delta} < \tau'_{\delta}\} \cup \{\tau_{\delta} = \tau'_{\delta}\}$, which is a disjoint union, and $\{\tau_{\delta} = \tau'_{\delta}\} = \{\tau_{\delta} = T\}$. Hence, $\{\tau_{\delta} < \tau'_{\delta}\} = \{\tau_{\delta} = T\}^c = \{\tau_{\delta} < T\}$. It follows that $P (\tau_{\delta} < \tau'_{\delta}) >0.$ \\
Set $\tilde{Y} = Y^{(1)} - Y^{(2)}$, then
\[d \tilde{Y}(t) = (f_2(t, Y^{(2)}(t), z) -f_1(t, Y^{(1)}(t), z)) dt.\]
Hence, by taking the integral of the above from $\tau_{\delta}$ to $\tau'_{\delta}$ and $\tilde{Y}(\tau_{\delta})=0$, we have
\begin{equation}\label{lem21_1}
\tilde{Y}(\tau'_{\delta})=Y^{(1)}(\tau'_{\delta})-Y^{(2)}(\tau'_{\delta})  \leq - \frac{\delta}{2} (\tau'_{\delta} - \tau_{\delta})\leq 0.
\end{equation}
Thus,
\[\dfrac{d\tilde{Y}(t)}{dt} \mathbb{I}_{\{t \in [\tau_{\delta}, \tau'_{\delta})\}}\leq - \frac{\delta}{2} \mathbb{I}_{\{t \in [\tau_{\delta}, \tau'_{\delta})\}}, ~~ \tilde{Y}(\tau_{\delta})=0.\]
So we deduce
\begin{equation} \label{lem21}
\tilde{Y}(\tau'_{\delta}) = Y^{(1)}(\tau'_{\delta})-Y^{(2)}(\tau'_{\delta}) \leq -\dfrac{\delta}{2} (\tau'_{\delta} - \tau_{\delta})< 0, ~~ \text{on} ~~ \{\tau_{\delta} < \tau'_{\delta}\}.
\end{equation}
Note, $(Y^{(i)}, z)$, $i=1,2$ are solutions of BSDEs with coefficients $(f_i, Y^{(i)}(T))$. It follows from Proposition \ref{f_exp_lem} (3), that
\[\mathcal{E}_{f_1}(Y^{(1)}(\tau'_{\delta})| \mathcal{F}_{\tau_{\delta}})=\mathcal{E}_{f_1}(\mathcal{E}_{f_1}(Y^{(1)}(T)| \mathcal{F}_{\tau'_{\delta}})|\mathcal{F}_{\tau_{\delta}}) = \mathcal{E}_{f_1}(Y^{(1)}(T)| \mathcal{F}_{\tau_{\delta}}) = y,\]
and similarly
\[\mathcal{E}_{f_2}(Y^{(2)}(\tau'_{\delta})| \mathcal{F}_{\tau_{\delta}}) = \mathcal{E}_{f_2}(Y^{(2)}(T)| \mathcal{F}_{\tau_{\delta}}) = y.\]
Moreover, again from Proposition \ref{f_exp_lem} (3),
\[\mathcal{E}_{f_1}(Y^{(1)}(\tau'_{\delta}))=\mathcal{E}_{f_2}(Y^{(2)}(\tau'_{\delta}))=y.\]
On the other hands, by \eqref{lem21_1} and \eqref{lem21}, we know
\[Y^{(1)}(\tau'_{\delta})\leq Y^{(2)}(\tau'_{\delta})\]
and
\[P(Y^{(1)}(\tau'_{\delta}) < Y^{(2)}(\tau'_{\delta})) >0.\]
It then follows from Definition \ref{f_expectation} and Proposition \ref{f_exp_lem} (2) that
\[y = \mathcal{E}_{f_1}(Y^{(1)}(\tau'_{\delta})) < \mathcal{E}_{f_1}(Y^{(2)}(\tau'_{\delta})),\]
but from i), we have
\[\mathcal{E}_{f_1}(Y^{(2)}(\tau'_{\delta})) \leq \mathcal{E}_{f_2}(Y^{(2)}(\tau'_{\delta})) =y,\]
which is a contradiction. So we conclude ii) holds.
\end{proof}




\end{document}